\documentclass[12pt]{article}

\usepackage{amssymb}
\usepackage{amsmath}
\usepackage{amsthm}
\usepackage{mathtools,xcolor,enumerate}
\usepackage[colorlinks=true,allcolors=black]{hyperref}
\usepackage[expansion=false]{microtype}

\newtheorem{theorem}{Theorem}
\newtheorem{proposition}{Proposition}[section]
\newtheorem{lemma}[proposition]{Lemma}

\theoremstyle{definition}

\theoremstyle{remark}

\newcommand{\R}{\mathbb{R}}
\newcommand{\C}{\mathbf{C}}
\newcommand{\cH}{\mathcal{H}}
\newcommand{\cL}{\mathcal{L}}
\newcommand{\Per}{\operatorname{Per}}
\newcommand{\diver}{\operatorname{div}}

\newcommand{\dist}{\operatorname{dist}}

\begin{document}

\title{A quantitative inequality for
general area-minimizing hypercones}
\author{Gongping Niu}
\date{}
\maketitle

\begin{abstract}
Let $\mathbf{C}=\partial E\subset\mathbb{R}^{n+1}$ be an
area-minimizing hypercone. The cone may have nonisolated
singularities. We prove that there is a constant
$c_{\mathbf{C}}>0$ such that
\[
 \frac{\operatorname{Per}(F;B_R)-\operatorname{Per}(E;B_R)}{R^n}
 \geq c_{\mathbf{C}}
 \left(\frac{|F\triangle E|}{R^{n+1}}\right)^2
\]
for every $R>0$ and every set $F$ of locally finite perimeter
with $F\triangle E\Subset B_R$.
This extends the unweighted quantitative inequality from regular
area-minimizing hypercones to general area-minimizing hypercones.

We follow the calibration argument in the author's earlier work
on regular cones. The main change is the construction of the
vector fields. The pointwise asymptotic estimates for positive
Jacobi fields used in the regular case do not directly apply here.
We use the minimal foliations constructed by Zhihan Wang on the
two sides of the cone. On each leaf, we integrate projection
kernels with a positive Jacobi field as the weight.
The kernel construction gives the required divergence bound.
Wang's weak Harnack inequality and growth estimates give the
integral bounds needed to prove linear growth of the vector fields.

\end{abstract}

\begingroup
\footnotesize
\endgroup


\section{Introduction}

An area-minimizing cone has no more perimeter than any competitor with
the same boundary data. A quantitative inequality asks how much the
perimeter must increase when the competitor differs from the cone.
We prove that the perimeter deficit
controls the square of this volume for every area-minimizing
hypercone. The cone may have singularities away from the origin.

De Philippis and Maggi \cite{de2014sharp} proved such inequalities for
area-minimizing Lawson cones, apart from six cases. Liu
\cite{liu2019stability} treated the remaining cases. These proofs use
the explicit geometry of the Lawson cones. In \cite[Theorem~1]{niu2026strict},
we proved the unweighted inequality for every regular area-minimizing
hypercone. Here regular means that the cone is smooth away from the
origin. That result does not require strict stability or strict
minimality. In this paper, our Theorem~\ref{thm:main} removes the regularity assumption.

Singular minimizing hypercones therefore satisfy the same quadratic
unweighted estimate as hyperplanes (with a constant depending on the cone).
For smooth solutions of the Plateau problem, De Philippis and Maggi
\cite[Theorem~1]{de2014sharp} prove a local quadratic estimate under
uniqueness and strict stability assumptions. Our result suggests that the same local quadratic inequality should hold
for singular Plateau solutions satisfying these two conditions.

\medskip

For a nonflat minimizing hypercone $\C\subset\R^{n+1}$, define
\[
 \operatorname{spine}\C
 =\{v\in\R^{n+1}:\C+v=\C\}.
\]
This is a linear subspace. After a rotation, the cone splits as
\[
 \C=\C_0\times\R^k,\qquad
 k=\dim\operatorname{spine}\C,\qquad
 \operatorname{spine}\C_0=\{0\}.
\]
The factor $\C_0$ is again an area-minimizing hypercone (\cite[Lemma~35.5]{simon1983lectures}). 
However, $\C_0$ need not be regular away from the origin.
Thus $\C$ need not be cylindrical, i.e., a product of a regular
minimizing hypercone and a Euclidean space. 
For every even integer $l\geq10$, Cui's recent work
\cite[Theorem~6.2]{cui2026laplacian} gives an area-minimizing
hypercone $\mathbf K\subset\R^{2 l}$ with
\[
  \operatorname{Sing}\mathbf K
  =(\R^l\times\{0\})\cup(\{0\}\times\R^l).
\]
Any translation preserving $\mathbf K$ also preserves its singular
set. This union of two coordinate subspaces has no nonzero translation
symmetry, so $\operatorname{spine}\mathbf K=\{0\}$.
Since $\mathbf K$ has singular points away from the origin,
it cannot be written as a regular minimizing hypercone times
$\R^j$.
Theorem~\ref{thm:main} applies to such cones and requires
no regularity assumption on the factor $\C_0$.

\medskip

The proof uses a minimal foliation on each side of the cone.
Hardt and Simon \cite{hardt1985area} constructed these foliations for
regular minimizing cones. On each side, the smooth minimizing leaf
is unique up to dilation. Wang \cite[Theorem~1.1]{wang2024mean}
proved the existence of smooth minimizing leaves for general minimizing
cones. His theorem does not establish uniqueness up to dilation in this
setting. Our proof only needs existence. We choose any leaf supplied by
Wang's theorem on each side and use its dilations to define the foliation.
The unit normal field $X$ is
a calibration. The calibration identity expresses the perimeter deficit as the integral of
$|\nu_F-X|^2/2$ over the competitor's reduced boundary (see also \eqref{eq:deficit} below). To control
the volume, we seek a field $Y$ on each side with
\[
 Y\cdot X=0,\qquad \diver Y\geq1,\qquad |Y(x)|\leq C|x|.
\]
A flux estimate and the Cauchy--Schwarz inequality then give the theorem, as in
\cite{niu2026strict}.

Hence the main task is to construct $Y$ with the stated bounds on the whole side.
On one leaf $\Sigma$, the function $\psi(z)=|z\cdot\nu(z)|$ is a
positive Jacobi field. The construction in \cite{niu2026strict}
uses the smooth link and pointwise estimates for this field at
infinity. For a general cone, the link itself may be
singular. The rescaled leaves converge smoothly only away from this
singular set. This convergence does not give uniform pointwise
estimates over the whole leaf. Thus the estimates used in the regular
case do not directly give the required bound here.

Jacob Bernstein suggested projecting the position vector onto the
tangent spaces of the foliation leaves. We use this idea to form
projection kernels centered at points of a fixed leaf.
We integrate these kernels over the leaf (with the positive Jacobi
field $\psi$ as the weight). After normalization, we extend the
resulting tangent field by dilation.

\medskip

In Section~\ref{sec:main} below we prove the inequality under the assumption of the existence of the desired tangent fields, so the arguments exactly follow the proof in \cite[section 3.1]{niu2026strict}. We will leave all the technical construction of the vector field $Y_\pm$ to Section~\ref{sec:construction} and show that the resulting field has the required properties.

\section{The main theorem}\label{sec:main}

Let $E\subset\R^{n+1}$ be an open cone. Assume that its boundary
$\C=\partial E$ is area minimizing.   
Set
\[
 U_-=E,\qquad U_+=\R^{n+1}\setminus\overline E,\qquad
 D_R(F)=\Per(F;B_R)-\Per(E;B_R).
\]
We orient $\C$ by the measure-theoretic outer normal $\nu_E$.

\begin{theorem}\label{thm:main}
There is a constant $c_\C>0$ with the following property.
Let $R>0$. Let $F\subset\R^{n+1}$ have locally finite perimeter and
satisfy $F\Delta E\Subset B_R$. Then
\begin{equation}\label{eq:main}
 \frac{D_R(F)}{R^n}
 \geq c_\C\left(\frac{\cL^{n+1}(F\triangle E)}{R^{n+1}}\right)^2.
\end{equation}
\end{theorem}

Theorem~\ref{thm:main} does not assume strict stability, strict minimality, or a smooth link. Hence it extends \cite[Theorem~1]{niu2026strict} from
regular minimizing hypercones to general minimizing hypercones.
In this section, we prove the inequality from the properties of two
vector fields. We construct these fields in
Section~\ref{sec:construction}.

\bigskip

Next we will introduce the idea of the proof. The strategy is exactly similar to \cite[Theorem 1]{niu2026strict}. We construct a calibration $X$ and a linear growth vector fields $Y_\pm$ which is orthogonal to $X$. Then the same argument as \cite[Proof of Theorem 1]{niu2026strict} gives the \eqref{eq:main}. 

By the Wang's theorem \cite[Theorem~1.1]{wang2024mean}, there is a smooth properly
embedded minimizing leaf $\Sigma_\pm$ on each side $U_\pm$ with $\dist(0,\Sigma_\pm)=1$ (different from \cite{hardt1985area}, we only know the existence, while the uniqueness is unknown, so we can just choose one such foliation).
Their dilations foliate $U_\pm$.
The map $\Phi_\pm(t,z)=tz$ is a diffeomorphism from
$(0,\infty)\times\Sigma_\pm$ onto $U_\pm$.
Choose their normals so that
$\pm z\cdot\nu_\pm(z)>0$. Define
\begin{equation}\label{eq:X}
     X(tz)=\nu_\pm(z),\qquad
 \lambda_\pm(tz)=t,\qquad
 \psi_\pm(z)=|z\cdot\nu_\pm(z)|>0.
\end{equation}

As $t\downarrow0$, the leaves $t\Sigma_\pm$ converge to $\C$ with
multiplicity one. The convergence is smooth on compact subsets of
$\operatorname{Reg}\C$.
Set $X=\nu_E$ on $\operatorname{Reg}\C$.
As in \cite[Lemma 4.1]{niu2026strict}, we claim that 
$$X\in C^0(\R^{n+1}\setminus\operatorname{Sing}\C; \mathbb{S}^n), \qquad  \qquad \diver X=0 \qquad \text{in} \quad  \mathcal D'(\R^{n+1}).$$

$X$ is smooth on $\R^{n+1}\setminus \C$. At a regular point on $\C$, the leaves converge smoothly, hence their normals converge  to $\nu_E$.
Thus $X$ is a continuous unit field on
$\R^{n+1}\setminus\operatorname{Sing}\C$.

The leaves are minimal and $|X|=1$, so $\diver X=0$ on each side.
The normal traces agree across $\operatorname{Reg}\C$.
Thus $\diver X=0$ also holds in the distributional sense outside
$\operatorname{Sing}\C$.
Fix a compact set. Since $\dim_{\cH}\operatorname{Sing}\C\leq n-7$,
we can choose cutoffs $\eta_j$ that vanish near the singular set and
tend to one away from it. They satisfy
$\int|\nabla\eta_j|\,dx\to0$.
Then by a standard capacity argument, 
\begin{equation}\label{eq:calibration}
 \diver X=0\quad\text{in }\mathcal D'(\R^{n+1}).
\end{equation}
Mollify $X$ and integrate against $D\chi_F-D\chi_E$.
This measure has compact support in $B_R$.
The field $X$ is continuous outside $\operatorname{Sing}\C$, and
$\cH^n(\operatorname{Sing}\C)=0$.
We can therefore pass to the limit by dominated convergence.
Since $X=\nu_E$ on $\partial^*E$, as in \cite[Lemma 4.3]{niu2026strict} (or \cite[Proposition 4.1]{de2014sharp}), we obtain
\begin{equation}\label{eq:deficit}
 D_R(F)=\int_{\partial^*F\cap B_R}(1-X\cdot\nu_F)\,d\cH^n
       =\frac12\int_{\partial^*F\cap B_R}|\nu_F-X|^2\,d\cH^n.
\end{equation}

\bigskip

Now we suppose that there are smooth fields $Y_\pm$ on $U_\pm$ satisfy
\begin{equation}\label{eq:field-assumptions}
 Y_\pm\cdot X=0,\qquad \diver Y_\pm\geq1,\qquad
 |Y_\pm(x)|\leq K_\C|x|\quad\text{on }U_\pm.
\end{equation}
Fix one side $U$. Write $Y=Y_\pm$ and $\lambda=\lambda_\pm$. Let
$A\subset U$ have finite perimeter with $A\Subset B_R$.
For almost every $s>0$, apply Gauss--Green to
$A_s=A\cap\{\lambda>s\}$.
The support of $A_s$ is compact in $U$.
The new boundary lies on $\{\lambda=s\}=s\Sigma$.
Its normal is parallel to $X$, so the $Y$-flux across it is zero.
Therefore
\[
 \cL^{n+1}(A_s)\leq\int_{A_s}\diver Y\,dx
 =\int_{\partial^*A\cap\{\lambda>s\}}Y\cdot\nu_A\,d\cH^n.
\]
Let $s\downarrow0$, the volumes converge by monotone convergence.
The boundary integrals converge by dominated convergence: $Y$ is
bounded on $B_R$, and $A$ has finite perimeter. We obtain
\begin{equation}\label{eq:volume-flux}
 \cL^{n+1}(A)\leq\int_{\partial^*A\cap U}Y\cdot\nu_A\,d\cH^n.
\end{equation}
This argument uses $Y$ only inside $U$.

\begin{proof}[Proof of Theorem~\ref{thm:main} from \eqref{eq:field-assumptions}]
Because the inequality \eqref{eq:main} is scaling invariant, we take $R=1$. Denote
$A_+=F\cap U_+$, and $A_-=F^c\cap U_-$.
Both $A_\pm$ have finite perimeter and compact support in $B_1$.
The outer normal of $A_+$ is $\nu_F$ inside $U_+$.
The outer normal of $A_-$ is $-\nu_F$ inside $U_-$.
Apply \eqref{eq:volume-flux} to each set and use $Y_\pm\cdot X=0$.
Then
\begin{align*}
 \cL^{n+1}(F\triangle E) &\leq\int_{\partial^*F\cap U_+}Y_+\cdot(\nu_F-X)\,d\cH^n
       -\int_{\partial^*F\cap U_-}Y_-\cdot(\nu_F-X)\,d\cH^n\\
  &\leq K_\C\int_{\partial^*F\cap B_1}|\nu_F-X|\,d\cH^n\\
  &\leq K_\C\bigl(2\Per(F;B_1)D_1(F)\bigr)^{1/2}.
\end{align*}
The last inequality uses Cauchy--Schwarz and \eqref{eq:deficit}.

If $D_1(F)\leq1$, then $[\cL^{n+1}(F\triangle E)]^2\leq2K_\C^2(\Per(E;B_1)+1)D_1(F)$.
If $D_1(F)>1$, then $[\cL^{n+1}(F\triangle E)]^2\leq\cL^{n+1}(B_1)^2D_1(F)$.
Thus \eqref{eq:main} is proved.
\end{proof}
Hence, the only technical work is the construction of the vector fields $Y_\pm$ (Proposition~\ref{prop:fields}). We leave it in section~\ref{sec:construction}.

\section{Construction of the vector fields \texorpdfstring{$Y_\pm$}{Y±}}\label{sec:construction}

The goal of this section is to construct the following vector fields.

\begin{proposition}\label{prop:fields}
On each open side $U_\pm$, there is a smooth field $Y_\pm$ such that
\begin{equation}\label{eq:goal}
  Y_\pm\cdot X=0,\qquad \diver Y_\pm\geq1,\qquad
  |Y_\pm(x)|\leq C_\C|x|,
\end{equation}
where $X$ is the normal field defined in \eqref{eq:X}.
\end{proposition}

\subsection{Ideas of the construction}\label{sec:kernel}

We fix on one side and write $\Sigma$ for its leaf and
$\psi(z):=|z\cdot\nu(z)|>0$.
We construct a tangent field $Z$ on $\Sigma$ and extend it by
\begin{equation}\label{eq:lift}
 Y(tz)=tZ(z).
\end{equation}
The map $\Phi(t,z)=tz$ has Jacobian $t^n\psi(z)$.
In these coordinates, $Y$ has components $(0,Z)$, so
\begin{equation}\label{eq:lift-divergence}
 (\diver Y)(tz)
 =\frac1{t^n\psi(z)}\diver_\Sigma(t^n\psi Z)(z)
 =\frac1{\psi(z)}\diver_\Sigma(\psi Z)(z).
\end{equation}
Thus it is enough to construct $Z$ with
\begin{equation}\label{eq:leaf-goal}
 Z(z)\in T_z\Sigma,\qquad
 \diver_\Sigma(\psi Z)\geq\psi,\qquad
 |Z(z)|\leq C |z|.
\end{equation}

We first try the tangential projection for intuition
\[
 P_0(x)=x-\langle x,X(x)\rangle X(x).
\]
On $\Sigma$, this gives $Z_0(z)=z^{\top_z}$.
Minimality gives $\diver_\Sigma Z_0=n$, and hence
\[
 \diver P_0(tz)=n+Z_0(z)\cdot\nabla_\Sigma\log\psi(z).
\]
The field is tangent to the leaves and satisfies $|P_0(x)|\leq|x|$. 
But we do not have a pointwise estimate for the last term that gives a uniform positive lower bound for its divergence (in the case where the cone is regular, one can indeed obtain a pointwise estimate for the last term from the asymptotic estimates in \cite{hardt1985area}. See also \cite[Lemma 4.2]{niu2026strict} for a similar approach). 

Instead of pointwise estimates for $\psi$, we construct \(Z\) by integrating projection kernels over \(\Sigma\) with weight \(\psi\). Integral estimates for \(\psi\) then give the required linear growth bound. Hence it avoids the need for uniform pointwise asymptotic estimates for \(\psi\). For $p,z\in\Sigma$ with $p\ne z$, set
\[
 P_p(z)=(z-p)^{\top_z},\qquad
 K_p(z)=\frac{P_p(z)}{|z-p|^n}.
\]
Since $\diver_\Sigma P_p=n$, we have
\begin{equation}\label{eq:kernel-div}
 \diver_\Sigma K_p(z)
 =\frac n{|z-p|^n}
  -\frac{n|(z-p)^{\top_z}|^2}{|z-p|^{n+2}}
 =\frac{n|(z-p)^{\perp_z}|^2}{|z-p|^{n+2}}\geq0.
\end{equation}
The kernel has size $O(|z-p|^{1-n})$, so it is locally integrable.

We integrate the kernels with weight $\psi(p)$ and set
\begin{equation}\label{eq:Z}
 Z(z):=\frac1{n\omega_n\psi(z)}
       \int_\Sigma\psi(p)K_p(z)\,d\cH^n(p).
\end{equation}
For now, assume that
\begin{equation}\label{eq:potential}
 \mathcal P_\psi(z):=
 \int_\Sigma\frac{\psi(p)}{|z-p|^{n-1}}\,d\cH^n(p)
 \leq C |z|\psi(z)
 \qquad(z\in\Sigma).
\end{equation}
We prove this estimate in Subsection~\ref{sec:estimates}.
Since $|K_p(z)|\leq|z-p|^{1-n}$, the integral in \eqref{eq:Z}
converges absolutely. Thus $Z$ is well-defined and tangent to $\Sigma$. In addition, we have the linear growth
\[
 |Z(z)|\leq\frac{\mathcal P_\psi(z)}{n\omega_n\psi(z)}
 \leq C|z|.
\]

\begin{lemma}\label{lem:kernel}
Under the assumption \eqref{eq:potential}, the field $Z$ in \eqref{eq:Z} is smooth.
\end{lemma}

\begin{proof}
Fix $z_0\in\Sigma$. Choose a smooth cutoff $\chi$ supported in
a small neighborhood of $z_0$, with $\chi=1$ near $z_0$.
For $z$ sufficiently close to $z_0$, write
\[
 \int_\Sigma\psi(p)K_p(z)\,d\cH^n(p)
 =
 \int_\Sigma\chi(p)\psi(p)K_p(z)\,d\cH^n(p)
 +
 \int_\Sigma(1-\chi(p))\psi(p)K_p(z)\,d\cH^n(p).
\]

For the first integral, use geodesic polar coordinates centered
at $z$. We choose the coordinate radius independently of $z$
near $z_0$. We use a change of variables, as in \cite[Section~2.2.1, Theorem~1]{evans2010partial}
for the Newtonian potential. In geodesic polar coordinates centered at $z$, the factor
$r^{1-n}$ in the kernel cancels the factor $r^{n-1}$
in the area element. The remaining integrand is smooth,
so differentiation under the integral gives local smoothness.

For the second integral, the cutoff removes the singularity
near $p=z$. The kernel $K_p(z)$ is smooth. For $z$ in a fixed compact set, all derivatives of order $k$ in $z$ are bounded
at infinity by $C_k|p|^{1-n}$.
These bounds are integrable with weight $\psi(p)$ by
\eqref{eq:potential} at a fixed point of $\Sigma$.
We can therefore differentiate the remaining integral under
the integral sign.
Hence, overall, $Z$ is smooth.
\end{proof}

We next prove
\begin{equation}\label{eq:kernel-size}
 \diver_\Sigma(\psi Z)\geq\psi.
\end{equation}

By \eqref{eq:kernel-div}, we have
\[
 \diver_\Sigma K_p
 =n\omega_n\delta_p
 +\frac{n|(z-p)^{\perp_z}|^2}{|z-p|^{n+2}}\,d\cH^n(z)
 \geq n\omega_n\delta_p
\]
in distributions. Let $\varphi\in C_c^\infty(\Sigma)$ be nonnegative, we have
\[
 -\int_\Sigma K_p(z)\cdot\nabla_\Sigma\varphi(z)\,d\cH^n(z)
 \geq n\omega_n\varphi(p).
\]
By multiplying this inequality by $\psi(p)/(n\omega_n)$
and integrating over $p\in\Sigma$, we obtain
\begin{align*}
 -\int_\Sigma\psi Z\cdot\nabla_\Sigma\varphi\,d\cH^n
 &=-\frac1{n\omega_n}\int_\Sigma\psi(p)
 \left(\int_\Sigma K_p(z)\cdot\nabla_\Sigma\varphi(z)
 \,d\cH^n(z)\right)d\cH^n(p)\\
 &\geq\int_\Sigma\psi(p)\varphi(p)\,d\cH^n(p).
 \end{align*}
Here we can exchange the two integrals because
$|K_p(z)|\leq|z-p|^{1-n}$ and \eqref{eq:potential}.
Since this holds for every nonnegative
$\varphi\in C_c^\infty(\Sigma)$, we have
$\diver_\Sigma(\psi Z)\geq\psi$ in distributions. Since $Z$ and $\psi$ are smooth, \eqref{eq:kernel-size} holds pointwisely.

\bigskip

Now extend $Z$ by \eqref{eq:lift}.
Tangency gives $Y\cdot X=0$, and \eqref{eq:lift-divergence} gives
$\diver Y\geq1$. Also, $|Y(tz)|\leq C |tz|$.
Applying the construction on both sides proves Proposition~\ref{prop:fields}.

\subsection{Wang's estimates}\label{sec:estimates}

We next prove \eqref{eq:potential} using Wang's estimates from \cite{wang2024mean}.
We first prove the following weak Harnack inequality for the positive Jacobi field $\psi$ on $\Sigma$.

\begin{lemma}
For every $1\leq q<n/(n-2)$, there is a constant $C_q$ such that
\begin{equation}\label{eq:Harnack}
 \left(
 \frac1{\cH^n(\Sigma\cap B_s(z))}
 \int_{\Sigma\cap B_s(z)}\psi^q\,d\cH^n
 \right)^{1/q}
 \leq C_q\inf_{\Sigma\cap B_s(z)}\psi
 \leq C_q\psi(z)
\end{equation}
for all $z\in\Sigma$ and $s>0$.
The constant $C_q$ is independent of $z$ and $s$.
Moreover, with $\Theta_\C=\cH^n(\C\cap B_1)$,
\begin{equation}\label{eq:densities}
     \omega_n s^n
 \leq\cH^n(\Sigma\cap B_s(z))
 \leq\Theta_\C s^n.
\end{equation}

\end{lemma}

\begin{proof}
Since the leaves $R^{-1}\Sigma$ blow down to $\C$ with multiplicity one, for all large $R>0$, the restrictions of $R^{-1}\Sigma$
to $B_4$ satisfy the assumptions of
\cite[Corollary~3.5]{wang2024mean}.

Let $\eta$ be the radius constant in \cite[Corollary~3.5]{wang2024mean}.
For any $z\in\Sigma$ and $s>0$, choose $R$ sufficiently large, we apply the corollary to $\psi(R\,\cdot)$ on $R^{-1}\Sigma$,
with center $z/R$ and radius $s/R$.
Rescaling gives \eqref{eq:Harnack}.

Finally, because the density $s^{-n}\cH^n(\Sigma\cap B_s(z))$ converges to $\Theta_\C$ for any fixed $z$ as $s\to\infty$, the monotonicity formula gives \eqref{eq:densities}.
\end{proof}
In the case $q=1$, by \eqref{eq:Harnack} and \eqref{eq:densities}, we get
\begin{equation}\label{eq:ball-integral}
 \int_{\Sigma\cap B_s(z)}\psi\,d\cH^n
 \leq C  s^n\psi(z).
\end{equation}

\paragraph{Decay of averages at infinity.}
For $r>1$, write
\[
 M(r)=\int_{\Sigma\cap B_r}\psi\,d\cH^n,\qquad
 I(r)=\frac{M(r)}{\cH^n(\Sigma\cap B_r)},\qquad
 \gamma_n=-\frac{n-2}{2}
 +\sqrt{\frac{(n-2)^2}{4}-(n-1)}.
\]
Since $n\geq7$, we have $\gamma_n<-1$.
Fix $\gamma_n<\gamma<-1$.
We first prove the following iteration lemma.

\begin{lemma}\label{lem:decay}
There is $R_0>1$ such that for $s\geq r\geq R_0$,
\begin{equation}\label{eq:decay}
 M(s)\leq C (s/r)^{n+\gamma}M(r).
\end{equation}
Consequently, $I(s)\leq C (s/r)^\gamma I(r)$.
\end{lemma}

\begin{proof}
For a positive Jacobi field $u$ on all of $\operatorname{Reg}\C$, set
\[
 I_u(\rho)=\frac1{\cH^n(\C\cap B_\rho)}
 \int_{\C\cap B_\rho}u\,d\cH^n.
\]
This is the $L^1_*$-norm in \cite{wang2024mean}.
By \cite[Lemma~A.2]{wang2024mean}, there is $a>0$, independent of $u$,
such that
\begin{equation}\label{eq:cone-growth}
 I_u(\rho)\geq a\rho^{\gamma_n}I_u(1)
 \qquad(0<\rho<1).
\end{equation}
Choose $0<\tau<1/2$ to be determined. 
We first show that for some small $\tau$,
\begin{equation}\label{eq:one-step}
 I(\tau R)\geq\tau^\gamma I(R)
 \quad\text{for every sufficiently large }R.
\end{equation}

Suppose this fails along $R_j\to\infty$.
Set $\Sigma_j=R_j^{-1}\Sigma$ and choose $p_j\in\Sigma_j$ tending
to a fixed regular point $p_0\in\C$.
Normalize by
\[
 u_j(y)=\frac{\psi(R_jy)}{\psi(R_jp_j)},\qquad u_j(p_j)=1.
\]
Applying \eqref{eq:Harnack} on larger balls containing $p_j$
gives uniform $L^q$ bounds on every fixed ball.
Smooth convergence of the leaves and local elliptic estimates give
a subsequence converging smoothly on compact subsets of
$\operatorname{Reg}\C$ to a nonnegative Jacobi field $u$, with $u(p_0)=1$.
The local Harnack inequality gives a positive lower bound near $p_0$.
Applying \eqref{eq:Harnack} on larger balls then gives a positive
lower bound on every fixed ball. Thus $u>0$ on all of
$\operatorname{Reg}\C$.

We next show that the assumption $ I(\tau R_j)<\tau^\gamma I(R_j)$
implies $I_u(\tau)\leq\tau^\gamma I_u(1)$ for the limiting
Jacobi field $u$ on $\C$. Fix $1<q<n/(n-2)$ and write $\mu_j=\cH^n\llcorner\Sigma_j,\mu=\cH^n\llcorner\C.$
Then by \eqref{eq:Harnack} and the area bound \eqref{eq:densities}, we have 
\[
 \int_{B_2}u_j^q\,d\mu_j\leq C.
\]
On each relatively compact subset of $\operatorname{Reg}\C\cap B_2$,
local smooth convergence of $\Sigma_j$ gives the same bound for the limit. By the exhaustion and monotone convergence theorem,
\[
 \int_{B_2}u^q\,d\mu\leq C.
\]
Thus, by H\"older's inequality, for every measurable $G\subset B_2$,
\[
 \int_Gu_j\,d\mu_j\leq C\mu_j(G)^{1-1/q},
 \qquad
 \int_Gu\,d\mu\leq C\mu(G)^{1-1/q}.
\]

Fix $b\in (0,2)$.
The singular set and $\partial B_b$ have zero $\mu$-measure.
For any $\varepsilon>0$, choose an open set
$V_\varepsilon\Subset B_2$ covering $(\operatorname{Sing}\C\cap\overline{B_b})\cup\partial B_b$ 
such that $\mu(V_\varepsilon)<\varepsilon$ and
$\mu(\partial V_\varepsilon)=0$.
Weak convergence of the area measures then gives
$\mu_j(V_\varepsilon)\to\mu(V_\varepsilon)$.
The estimates above therefore imply
\[
 \int_{V_\varepsilon}u_j\,d\mu_j
 +\int_{V_\varepsilon}u\,d\mu
 \leq C\varepsilon^{1-1/q}
\]
for all sufficiently large $j$.

Outside $V_\varepsilon$, we are away from the singular set.
Local smooth convergence and $\mu(\partial V_\varepsilon)=0$ give
\[
 \int_{B_b\setminus V_\varepsilon}u_j\,d\mu_j
 \longrightarrow
 \int_{B_b\setminus V_\varepsilon}u\,d\mu.
\]
Combining the two parts, we obtain
\[
 \limsup_{j\to\infty}
 \left|
 \int_{B_b}u_j\,d\mu_j-\int_{B_b}u\,d\mu
 \right|
 \leq C\varepsilon^{1-1/q}.
\]
Letting $\varepsilon\downarrow0$ proves convergence of the integrals.
By rescaling, we obtain
\[
 \frac{I(bR_j)}{\psi(R_jp_j)}
 =
 \frac{\int_{B_b}u_j\,d\mu_j}{\mu_j(B_b)}
 \longrightarrow
 \frac{\int_{B_b}u\,d\mu}{\mu(B_b)}
 =I_u(b).
\]
Our assumption $I(\tau R_j)<\tau^\gamma I(R_j)$ then gives
$I_u(\tau)\leq\tau^\gamma I_u(1)$.
But \eqref{eq:cone-growth} implies
\[
 I_u(\tau)\geq a\tau^{\gamma_n}I_u(1).
\]
So as long as we choose $\tau$ small enough so that $a\tau^{\gamma_n}>2\tau^\gamma$, this contradicts $I_u(\tau)\leq\tau^\gamma I_u(1)$. This proves \eqref{eq:one-step}.

Choose $R_0$ large enough so that \eqref{eq:one-step} holds for
$R\geq R_0$. Given $s\geq r\geq R_0$, choose $j\geq0$ with
$t=\tau^js\in(\tau r,r]$.
Iteration gives $I(s)\leq(s/t)^\gamma I(t)$.
The area bounds \eqref{eq:densities} and $M(t)\leq M(r)$ yield
\[
 M(s)\leq C (s/t)^{n+\gamma}M(t)
 \leq C (s/r)^{n+\gamma}M(r).
\]
\end{proof}

\begin{proof}[Proof of \eqref{eq:potential}]
Fix $z\in\Sigma$ and set $a=\max\{R_0,|z|\}$.
For $|p|\leq2a$, we have $|p-z|\leq3a$.
Set $r_j=3a\,2^{-j}$ and apply \eqref{eq:ball-integral}:
\[
 \int_{\Sigma\cap B_{2a}}
 \frac{\psi(p)}{|z-p|^{n-1}}\,d\cH^n(p)
 \leq\sum_{j\geq0}r_{j+1}^{1-n}
 \int_{\Sigma\cap B_{r_j}(z)}\psi\,d\cH^n
 \leq C \psi(z)\sum_{j\geq0}r_j
 \leq C  a\psi(z).
\]
For $|p|>2a$, we have $|p-z|\geq|p|/2$.
Summing over shells centered at the origin and using \eqref{eq:decay},
we obtain
\[
 \int_{\Sigma\setminus B_{2a}}
 \frac{\psi(p)}{|z-p|^{n-1}}\,d\cH^n(p)
 \leq C  a^{1-n}M(a)\sum_{j\geq1}2^{j(1+\gamma)}
 \leq C  a\psi(z).
\]
The series converges because $\gamma<-1$.
Here $M(a)\leq C  a^n\psi(z)$ follows from
$B_a\subset B_{2a}(z)$ and \eqref{eq:ball-integral}.
Finally, $|z|\geq1$ gives $a\leq R_0|z|$.
This proves \eqref{eq:potential}.
\end{proof}

The construction in Subsection~\ref{sec:kernel} now proves
Proposition~\ref{prop:fields}.

\section*{Acknowledgments}

I would like to thank Jacob Bernstein for his questions and suggestions during my
visit to Johns Hopkins University. He suggested projecting the position
vector onto the tangent spaces of the foliation leaves. This idea was
the starting point for the vector field construction in this paper. I would also like to thank Zhihan Wang for proposing this question to me and for his explanation of his previous work.

\section*{Disclosure of AI use}

I developed the proof with assistance from AI (OpenAI's GPT-5.6 and GPT-6.0). I used these models to explore proof ideas, work out details,
check arguments, and revise the English. I formulated the problem,
identified the relevant literature, and chose the proof strategies.
I take responsibility for the mathematical claims and the final text.

\paragraph{Development of the proof.}
My initial approach followed the PDE method in my earlier work \cite{niu2026strict}
on regular minimizing hypercones. I tried to construct the vector
fields by solving a PDE on a foliation leaf. I discussed this
approach with AI several times, but these attempts
did not lead to a proof. The approach seemed to require uniform
pointwise asymptotic estimates for the positive Jacobi field on
the whole leaf. I could not establish these estimates for general
minimizing hypercones.

During my visit to Johns Hopkins University, Jacob Bernstein
suggested projecting the position vector onto the tangent spaces
of the foliation leaves. This field is tangent to the leaves and
has at most linear growth. His suggestion led me to reconsider
the divergence condition. I had focused on obtaining the exact
identity $\diver Y=1$. I then realized that the proof only needs
a uniform positive lower bound for $\diver Y$, together with
tangency and linear growth. In the regular case, Bernstein's
suggestion gives a simpler construction after a minor modification.

I then explored whether the same idea could work for general
minimizing hypercones. The original projection field still
requires pointwise control of the Jacobi field to estimate its
divergence. I discussed this difficulty and the weaker divergence
condition with AI. In these discussions, AI suggested integrating
translated projection kernels with a weight. I developed this
construction through further discussions with AI, using the
positive Jacobi field as the weight.

The change in the divergence condition was important in this process. Bernstein's suggestion led me to drop the exact identity
and explore a different construction. The later discussions with
AI then led to a weighted integral construction that avoids the
pointwise asymptotic estimates needed in the earlier PDE approach.

{\small
\let\originalthebibliography\thebibliography
\renewcommand{\thebibliography}[1]{%
  \originalthebibliography{#1}%
  \setlength{\itemsep}{0pt}%
  \setlength{\parsep}{0pt}%
  \setlength{\parskip}{0pt}%
}
\bibliographystyle{alpha}
\bibliography{main}
}

\bigskip
\begingroup
\small
\normalfont
\noindent
Department of Mathematics, University of Rochester\\
Rochester, New York\\
Email address:
\href{mailto:gniu3@ur.rochester.edu}
     {\texttt{gniu3@ur.rochester.edu}}
\par
\endgroup
\end{document}